\documentclass[11pt]{amsart}
\title{ On non-primitive Weierstrass points }
\author[N. Pflueger]{Nathan Pflueger}\address{Department of Mathematics, Brown University, Box
1917, Providence, RI 02912}\email{pflueger@math.brown.edu}
\usepackage{datetime}
\newdate{date}{5}{1}{2018}
\date{\displaydate{date}}

\usepackage{setspace}
\usepackage{color}
\usepackage{xifthen}
\usepackage{amssymb}
\usepackage{enumerate}
\usepackage{tikz}

\newtheorem{thm}{Theorem}[section]
\newtheorem{lemma}[thm]{Lemma}
\newtheorem{prop}[thm]{Proposition}

\newtheorem{cor}[thm]{Corollary}

\newtheorem{conj}[thm]{Conjecture}

\theoremstyle{definition}
\newtheorem{defn}[thm]{Definition}
\newtheorem{innereg}[thm]{Example}
\newenvironment{eg}
  {\begin{innereg}}
  {\end{innereg}}

\newtheorem{rem}[thm]{Remark}
\newtheorem{sit}[thm]{Situation}
\newtheorem{qu}[thm]{Question}

\newcommand{\NN}{\textbf{N}}
\newcommand{\PP}{\textbf{P}}
\newcommand{\ZZ}{\textbf{Z}}

\newcommand{\iou}[1][]{
    \ifthenelse{\equal{#1}{}}{{\color{blue}\{IOU\}}}
    {{\color{blue}\{IOU: #1\}}}
}

\newcommand{\cC}{\mathcal{C}}

\newcommand{\cG}{\mathcal{G}}
\newcommand{\cM}{\mathcal{M}}
\newcommand{\cO}{\mathcal{O}}

\newcommand{\cX}{\mathcal{X}}
\newcommand{\msg}{\cM^S_{g,1}}

\DeclareMathOperator{\Pic}{Pic}
\DeclareMathOperator{\wt}{wt}
\DeclareMathOperator{\ew}{ewt}
\DeclareMathOperator{\slide}{slide}
\DeclareMathOperator{\rf}{ref}
\DeclareMathOperator{\val}{val}
\DeclareMathOperator{\chr}{char}
\DeclareMathOperator{\codim}{codim}
\DeclareMathOperator{\Spec}{Spec}
\DeclareMathOperator{\NGS}{NG}

\begin{document}
\maketitle

\begin{abstract}
We give an upper bound for the codimension in $\cM_{g,1}$ of the variety $\msg$ of marked curves $(C,p)$ with a given Weierstrass semigroup. The bound is a combinatorial quantity which we call the effective weight of the semigroup; it is a refinement of the weight of the semigroup, and differs from the weight precisely when the semigroup is not primitive. We prove that whenever the effective weight is less than $g$, the variety $\msg$ is nonempty and has a component of the predicted codimension. These results extend previous results of Eisenbud, Harris, and Komeda to the case of non-primitive semigroups. We also survey other cases where the codimension of $\msg$ is known, as evidence that the effective weight estimate is correct in wider circumstances.
\end{abstract}


\section{Introduction}

Given a point $p$ on a smooth curve $C$ of genus $g$, there is an associated numerical semigroup
$$
S(C,p) = \left\{ - \val_p(f):\ f \in \Gamma(C \backslash \{p\}, \cO_C ) \right\},
$$
given by the pole orders of rational functions with no poles away from $p$. Weierstrass's \emph{L\"uckensatz} (now an easy consequence of the Riemann-Roch formula) states that there are exactly $g$ gaps in $S(C,p)$\footnote{The author has heard conflicting stories about whether the number of gaps in a numerical semigroup is called the ``genus'' due to this fact from geometry, or as a joking reference to the ``number of holes'' in the semigroup.}.

In reverse, any numerical semigroup $S$ with $g$ gaps defines a (not necessarily closed) subvariety $\msg \subseteq \cM_{g,1}$ of the moduli space of curves with a marked point. These loci stratify $\cM_{g,1}$, with the locus defined by the \emph{ordinary} semigroup $H_g = \{0,g+1,g+2,\cdots\}$ dense and open, and the value of the $i$th gap ($i=1,2,\cdots,g$) an upper semicontinuous function.

The link between the combinatorics of these numerical semigroups and the geometry of curves and their moduli is a wide and fascinating story that remains largely mysterious, though many intriguing special cases (specific types of semigroups) are well-understood. The core of the difficulty (and excitement) in this story lies in the fact that $S(C,p)$ is not an arbitrary sequence of integers, but a semigroup; this combinatorial restriction reflects itself in the geometry of the stratification.

Our objective is to propose a partial answer to a basic question: given a semigroup $S$, what is the codimension of $\msg$ in  $\cM_{g,1}$? 

\begin{defn}
The \emph{effective weight} of a numerical semigroup $S$ is 
$$
\ew(S) = \sum_{\textrm{gaps } b} \left( \# \mbox{ generators $a < b$} \right).
$$
Alternatively, $\ew(S)$ is the number of pairs $(a,b)$, where $0 < a < b$, $a$ is a generator, and $b$ is a gap.
\end{defn}

In almost every situation where $\codim \msg$ is known for an explicit family of semigroups (as well as for all semigroups of genus up to $6$), it is equal to $\ew(S)$; we summarize a number of these cases in Section \ref{s_background}. The first genus in which the author is aware of a semigroup with $\codim \msg < \ew(S)$ is $g=9$ (the example is discussed in Section \ref{ss_notew}). 

Our main results are the following, which give much stronger evidence for the utility of $\ew(S)$ in the study of this stratification of $\cM_{g,1}$.

\begin{thm} \label{t_ewbound}
If $\msg$ is nonempty, and $X$ is any irreducible component of it, then $$\dim X \geq \dim \cM_{g,1} - \ew(S).$$
\end{thm}

We call a point or irreducible component of $\msg$ \emph{effectively proper} if the local dimension of $\msg$ is exactly $\dim \cM_{g,1} - \ew(S)$.

\begin{thm} \label{t_ewexist}
If $S$ is a genus $g$ numerical semigroup with $\ew(S) \leq g-2$, then $\msg$ has an effectively proper component. If $\chr k = 0$, then the same is true for all numerical semigroups with $\ew(S) \leq g-1$.
\end{thm}

The effective weight is a refinement of a more naive quantity, the \emph{weight} of a semigroup, and the two quantities are equal for $S$ if and only if $S$ is \emph{primitive}, meaning that the sum of any two nonzero elements is greater than the largest gap (see Section \ref{ss_ehk}).

Theorems \ref{t_ewbound} and \ref{t_ewexist} were originally proved, with a characteristic $0$ hypothesis, for primitive semigroups (using the weight) by Eisenbud, Harris, and Komeda \cite{eh87,kom91}. Our proofs are based on theirs, using the theory of limit linear series as the central technical tool. Our primary innovation is to apply the machinery of limit linear series to produce \emph{incomplete} linear series with specified vanishing data on smooth curves. The basic technique is the same: curves with Weierstrass semigroups of genus $g$ are constructed by choosing a suitable genus $g-1$ semigroup and a marked curve realizing it, attaching an elliptic curve at the Weierstrass point, marking a second point on the elliptic curve differing by torsion, and deforming the resulting nodal curve.

\begin{rem} \label{r_st}
The choice of terminology ``effective weight'' was made in reference to terminology from the numerical semigroup literature. The set of all numerical semigroups can be arranged in a rooted tree, with each level corresponding to a different genus, where the parent of a semigroup $S$ is given by adding the largest gap back into $S$. The children of a given semigroup $S$ correspond to the ``effective generators'' of $S$, which are defined to be the generators that are larger than the largest gap. For details, and a study of the structure of this tree, see \cite{bras}. The effective weight of $S$ is determined by examining, in the path from the root (genus $0$ semigroup) to $S$, the index of the effective generator removed at each step (when the effective generators are listed in increasing order). If a similar procedure were followed, arranging all cofinite subsets of $\NN$ into a tree (not just semigroups), then the quantity constructed in the same way would be the weight, rather than the effective weight.
\end{rem}

\subsection{Speculation and conjectures}

While there are examples of semigroups $S$ for which $\codim \msg < \ew(S)$, to the author's knowledge all such examples fall in the range $g \leq \codim \msg \leq 2g$. Therefore, we (somewhat speculatively) conjecture that no such semigroups exist in codimension less than $g$.

\begin{conj}
If $\msg$ has a component of codimension less than $g$ in $\cM_{g,1}$, then all components of $\msg$ have codimension exactly $\ew(S)$.
\end{conj}

Curiously, we are not aware of any numerical semigroups of any genus for which $\codim \msg > 2g$. Therefore we also make the following (equally speculative) conjecture.

\begin{conj}
For any numerical semigroup such that $\msg \neq \emptyset$, all components of $\msg$ have codimension at most $2g$.
\end{conj}

Note that in the above conjectures, $g$ and $2g$ perhaps ought to be replaced with $g+C_1$ and $2g+C_2$ for some constants $C_1,C_2$, the value of which we have no strong beliefs about. We have stated the conjectures as above merely to make them specific.

Although not relevant to the present paper, we also mention a purely combinatorial conjecture about the effective weight that arose during this work. We have verified this conjecture by a computer search\footnote{Source code in C++ is available on the author's website. The search took approximately 17 hours on a 3.4Ghz Intel i7-3770 CPU.} up to genus $50$.

\begin{conj} \label{conj_maxew}
For any numerical semigroup of genus $g$, $$\ew(S) \leq \left\lfloor \frac{(g+1)^2}{8} \right\rfloor.$$
\end{conj}

\begin{rem}
If true, this conjecture is sharp. For $g \leq 5$, this follows from case analysis. For $g \geq 6$, this follows from a general construction. Let $\eta$ be an integer between $-2$ and $2$ inclusive such that $\eta \equiv g+1 \pmod{4}$ (there are two choices if $g \equiv 1 \pmod{4}$, and one otherwise). Let $c = \frac14 (3g+3+\eta)$ and $d = \frac14(5g+1+3\eta)$. Then the semigroup
\begin{eqnarray*}
S &=& \langle c,c+1,\cdots,d-1,d \rangle\\
&=& \NN \backslash \left\{ 1,2,\cdots,c-1,\ d+1,d+2,\cdots, 2c-1 \right\}
\end{eqnarray*}
has genus $g$ and effective weight $\frac18 (g+1)^2 - \frac18 \eta^2 = \left\lfloor \frac18 (g+1)^2 \right\rfloor$. For $10 \leq g \leq 50$, a computer search shows that these are the only semigroups of this effective weight, while for $g \leq 9$ there are some additional sporadic examples achieving the same maximum.
\end{rem}

The semigroups above (that appear, empirically, to maximize $\ew(S)$ in a given genus) also provide examples where $\codim \msg < \ew(S)$ \cite{pfl}.

\section*{Outline of the paper}

We summarize in Section \ref{s_background} several cases where $\codim \msg$ is known in the literature, including the simplest case where strict inequality $\codim \msg < \ew(S)$ occurs. Section \ref{s_dp} summarizes background on linear series and limit linear series needed for the proofs of the main theorems. Theorem \ref{t_ewbound} is proved in Section \ref{s_bound}. Section \ref{s_secundive} is purely combinatorial, and provides some preliminary results on the structure of numerical semigroups of low effective weight. Section \ref{s_exist} gives the proof of Theorem \ref{t_ewexist}.

\section*{Conventions}

Throughout this paper, we work over an algebraically closed field $k$. In Section \ref{s_background}, we assume $\chr k = 0$. A \emph{point} of a scheme will always refer to a closed point, and when we say that a \emph{general point} of a scheme satisfies a property, we mean that there exists a dense open subset in which all points satisfy the property. A \emph{curve} is always reduced, connected, and complete. A \emph{marked curve} is a pair $(C,p)$ of a curve $C$ and a point $p \in C$.

We denote by $\NN$ the set of nonnegative integers; a \emph{numerical semigroup} is a cofinite subset $S \subseteq \NN$ containing $0$ and closed under addition. The elements of $\NN \backslash S$ are called the \emph{gaps} of $S$, and the number of gaps is called the \emph{genus}. A positive element of $S$ that is not equal to the sum of two positive elements of $S$ is called a \emph{generator}, and a sum of two positive elements is called \emph{composite}.

We denote the set $\{0,g+1,g+2,\cdots\}$ by $H_g$, which we call the \emph{ordinary semigroup of genus $g$}.

\section*{Acknowledgments}

Most of this paper was developed from part of my thesis at Harvard University, and I owe Joe Harris a great debt for suggesting the questions and providing several key insights. I also thank Nathan Kaplan for a number of helpful conversations on numerical semigroups, and the anonymous referee for detailed comments on the manuscript.

\section{Background} \label{s_background}

The classification question of Weierstrass points can be asked on various levels. Hurwitz \cite{hur} first raised the simple existence question, while we are concerned with the more geometric dimension question.

\begin{qu} \label{q_exist}
For which $S$ is $\msg \neq \emptyset$?
\end{qu}
\begin{qu} \label{q_allcd}
Given $S$, how many irreducible components does $\msg$ have, and what are their codimensions?
\end{qu}
\begin{qu} \label{q_maxcd}
Given $S$, what is the maximum codimension of an irreducible component of $\msg$?
\end{qu}

The number of semigroups of genus $g$ grows exponentially with $g$ with limiting ratio $\frac{1+\sqrt{5}}{2}$ \cite{zhai}, and present knowledge, even about Question \ref{q_exist}, becomes quite sparse for large genus if all semigroups are considered (see \cite{kaplanye}). This is one reason we prefer to focus on Questions \ref{q_allcd} and \ref{q_maxcd}: if one hopes for general results, matters become much more tractable upon restricting to the more plentiful sorts of semigroups, i.e. those for which $\codim \msg$ is small compared to $g$.

This restriction, to semigroups expected to appear in low codimension, is what allowed Eisenbud and Harris to prove their rather strong results, later extended by Komeda. The downside of their results is that they needed to impose not just a quantitative restriction (weight being less than $g$) but a qualitative one: that the semigroup is primitive.

In the remainder of this section, we summarize some known results and simple cases of answers to these questions, in order to highlight the extent to which the effective weight brings many known cases under one umbrella. We conclude in \ref{ss_notew}, however, with the first example we know in which the effective weight does not give the correct codimension.

Throughout this background section, we will assume that $\chr k = 0$, as much of the literature makes this assumption.

\subsection{The work of Eisenbud, Harris and Komeda} \label{ss_ehk}

The \emph{weight} of a numerical semigroup is most simply defined as the sum of the gaps minus $\binom{g+1}{2}$. An alternate description, more suggestive of the link to the effective weight, is that $\wt(S)$ is the number of pairs $(a,b)$ where $0<a<b$, $a \in S$, and $b \not\in S$. This description shows that $\wt(S) - \ew(S)$ is equal to the number of pairs $(a,b)$ where $a<b$, $a$ is \emph{composite}, and $b$ is a gap; hence $\wt(S) = \ew(S)$ if and only if $S$ is primitive.

All semigroups satisfy $\codim \msg \leq \wt(S)$ (see Remark \ref{rem_wtbound} for one argument), and a point of $\msg$ at which equality holds locally is called \emph{dimensionally proper}. Eisenbud and Harris \cite{eh87} proved that if $S$ is primitive and $\wt(S) \leq g-2$, then $\msg$ has dimensionally proper points. Their proof made a characteristic $0$ assumption, since this assumption was built into their theory of limit linear series developed in \cite{eh86}, but modern treatments of limit linear series (e.g. \cite{oss06} or \cite{ossBook}) make no such assumption. The proofs of \cite{eh87} can therefore be carried to characteristic $p$ with no modification.

The argument of \cite{eh87} proceeds by induction on $g$. It nearly succeeds in proving the same result for primitive semigroups with $\wt(S) \leq g-1$ (rather than $g-2$), except that the inductive step fails for one very specific class of semigroups of weight $g-1$. Komeda's contribution \cite{kom91} is to prove the theorem in this special case by a different argument, without limit linear series (and with a characteristic $0$ hypothesis), thus extending the results of \cite{eh87} to the case $\wt(S) = g-1$. A second argument for this special case appears in \cite{ck}.

Eisenbud and Harris observe \cite[Corollary on p. 497]{eh87} that the primitivity hypothesis is necessary, i.e. $\codim \msg < \wt(S)$ if $S$ is non-primitive. This fact of course also now follows from our Theorem \ref{t_ewbound}. This is no minor difficulty, as many semigroups, including those that appear with low codimension in the Weierstrass stratification of $\cM_{g,1}$, are not primitive. The main example, which provided substantial motivation regarding how to refine $\wt(S)$, is the following.

\begin{eg}
A hyperelliptic curve of genus $g$ has $2g+2$ points with semigroup $\{2,4,6,\cdots,2g-2\} \cup H_{2g} = \langle 2,2g+1 \rangle$ (the ramification points of the double cover of $\PP^1$), while the rest of the points have the ordinary semigroup (see e.g. \cite[exercise I.E-3]{acgh}). Furthermore, if $2 \in S(C,p)$ then $C$ is necessarily hyperelliptic. Hence $S = \langle 2, 2g+1 \rangle$ is called the \emph{hyperelliptic semigroup}, and $\codim \msg = g-1$ (the codimension of the hyperelliptic locus in $\cM_{g}$ plus $1$).

The hyperelliptic semigroup has the distinction of having the maximum weight of all genus $g$ semigroups, namely $\binom{g}{2}$. So the weight bound is spectacularly off in this case. However $\ew(S) = g-1$.
\end{eg}

\begin{rem}
Since the semigroups of maximum weight provide a nice example where the weight bound fails to be exact (and suggested the definition of the effective weight), it seems reasonable to try to find cases where the effective weight bound fails to be exact in the semigroups of maximum effective weight. Indeed, these semigroups provide such examples; see \cite{pfl}. See also Conjecture \ref{conj_maxew} and the remark following it.
\end{rem}

One notable extension of Eisenbud and Harris's result, and method of proof, was given by Bullock \cite{bullock}. Using a variation on Eisenbud and Harris's inductive argument, Bullock proves that for the non-primitive semigroup
$$S = \{0,g-1,g+1,g+2,\cdots,2g-2\} \cup H_{2g}$$
of weight $g$, the locus $\msg$ is irreducible of codimension $g-1$. The manner in which Bullock treated a non-primitive semigroup with Eisenbud and Harris's basic method provided inspiration for our method in proving the more general Theorem \ref{t_ewexist}. Note that $S$ is ``barely non-primitive,'' as there is only one gap exceeding one composite element. See Remark \ref{r_ngsym} for more about Bullock's work.

\subsection{The Deligne bound and negatively graded semigroups} \label{ss_deligne}

The best general-purpose \textit{lower} bound on $\codim \msg$ is the Deligne bound, defined below.

\begin{defn}
For any numerical semigroup $S$, let $\lambda(S)$ be the number of gaps $b \not\in S$ such that $b + a \in S$ for all positive elements $a \in S$.
\end{defn}

\begin{prop}
Let $S$ be a numerical semigroup of genus $g$. If $\msg$ is nonempty, then $\codim \msg \geq g - \lambda(S)$.
\end{prop}
\begin{proof}
This bound follows from results of Deligne \cite{deligne}, first applied to the the moduli of Weierstrass points by Pinkham \cite[Theorems 10.3, 13.9]{pinkham}. For a discussion of the bound in this form, see \cite[Corollary 6.3]{rv}.
\end{proof}

In most cases, the Deligne bound and the effective weight bound do not coincide. Interestingly, the cases where they do coincide are semigroups of a particular structure: they are the ``negatively graded semigroups'' studied by Rim and Vitulli. Rim and Vitulli prove that a semigroup is negatively graded (a deformation-theoretic condition) if and only if it is one of the following \cite[Theorem 4.7]{rv}.

\begin{defn}
Let $g \geq 2$ be a positive integer. For each integer $e$ between $1$ and $g-1$ inclusive, define:
\begin{eqnarray*}
\NGS^1_{g,e} &=& (g-e+1) \cdot \ZZ \cup H_c\\
&& \mbox{where } c = g+\left\lfloor \frac{g}{g-e} \right\rfloor.\\
\NGS^2_{g,e} &=& \{0,g,g+1,\cdots,g+e-1\} \cup H_{g+e}.\\
\end{eqnarray*}
and also define, for $g \geq 3$,
$$\NGS^3_{g} = \{0,g-1,g+1,g+2,\cdots,2g-2\} \cup H_{2g-1}.$$
\end{defn}

Observe that $\ew(\NGS^1_{g,e}) = \ew(\NGS^2_{g,e}) = e$, and $\ew(\NGS^3_g) = g-1$. Note that $\NGS^1_{g,1} = \NGS^2_{g,1}$ and $\NGS^1_{3,2} = \NGS^3_3$, but in all other cases the semigroups described above are distinct. Therefore for $g \geq 4$, there is one negatively graded semigroup of effective weight $1$, two negatively graded semigroups of effective weight $e$ for $2 \leq e \leq g-2$, and three negatively graded semigroups of effective weight $g-1$. Half of the semigroups $\NGS^1_{g,e}$ (those for which $e \leq \frac{g}{2}$), and all of the semigroups $\NGS^2_{g,e}$ are primitive, while $\NGS^3_g$ and the other half of the $\NGS^1_{g,e}$ are not.

\begin{rem} \label{r_ngsym}
The three semigroups $\NGS^1_{g,g-1}$, $\NGS^2_{g,g-1}$, $\NGS^3_g$ of effective weight $g-1$ were studied by Bullock in \cite{bullock}; they correspond to the three irreducible components of the locus $\{(C,p) \in \cM_{g,1}:\ (2g-2)p \sim K_C \}$ of ``subcanonical points;'' see \cite{kz} for further background about this locus. All three are called \emph{symmetric} semigroups, since a positive integer $n$ is a gap if and only if $2g-1-n$ is not a gap; this condition is equivalent to the condition that $2g-1$ is a gap.
\end{rem}

\begin{rem}
The semigroups $\NGS^2_{g,e}$ are among the first semigroups for which $\msg$ was studied in detail; see \cite[Theorem 14.7]{pinkham}. 
\end{rem}

In fact, the negatively graded semigroups are precisely the semigroups for which the Deligne lower bound (on codimension) and the effective weight upper bound coincide.

\begin{prop}
For any numerical semigroup $S$ of genus $g$, $$\ew(S) \geq g - \lambda(S),$$ with equality if and only if $S$ is either ordinary or one of the semigroups $\NGS^1_{g,e}, \NGS^2_{g,e}$, or $\NGS^3_g$.
\end{prop}
\begin{proof}
Let $E$ denote the set of pairs $(a,b) \in \NN^2$ such that $a<b$, $a$ is a generator of $S$, and $b$ is a gap. Let $\Lambda$ denote the set of gaps $b$ such that $b + a \in S$ for all positive elements $a \in S$. By definition, $\ew(S) = |E|$ and $\lambda(S) = |\Lambda|$.

For all $(a,b) \in E$, $b-a$ is necessarily a gap that is not in $\Lambda$. Conversely, any gap $b'$ that is not an element of $\Lambda$ must be equal to $b-a$ for some $(a,b) \in E$. This shows that the complement of $\Lambda$ in $\NN \backslash S$ has at most $|E|$ elements, hence $\ew(S) \geq g - \lambda(S)$. Furthermore, this argument shows that equality holds if and only if each $(a,b) \in E$ gives a \emph{distinct} difference $b-a$. Assume now that $S$ satisfies $\ew(S) = g - \lambda(S)$; we will show that $S$ is of one of the three forms stated. The case where $S$ is ordinary is immediate, so assume that $S$ is non-ordinary. Denote by $m,n$ the first two generators of $S$.

\textit{Case 1:} Suppose there are no gaps above $n$. In this case $S = \NGS^1_{g,g-m+1}$.

\textit{Case 2:} Suppose that $n = m+1$. There can be no two consecutive gaps $b,b+1$ of $S$ greater than $m+1$, since otherwise $(m,b)$ and $(m+1,b+1)$ both lie in $E$. Similarly, there is at most one gap $b$ such that $b-1$ is a generator. Since all elements of $S$ less than $2m$ are generators, these two facts show that there is at most one gap $b$ of $S$ between $m$ and $2m$. If there are no gaps between $m$ and $2m$, then $S$ is ordinary. If there is one gap $b$ between $m$ and $2m$, then $S$ contains $\{m,m+1,\cdots,b-1, b+1,b+2,\cdots,2m-1\}$, which generate all integers greater than $2m$ (recall that $b > n = m+1$ by assumption), so in fact $b$ is the only gap greater than $m$. Hence $S = \NGS^2_{g,b-g}$.

\textit{Case 3:} Suppose that $n \geq m+2$ and there is some gap $b > n$. Assume that $b$ is the smallest such gap. The gap $b$ is less than $m+n$, since otherwise $b-m$ would be an element of $S$ and $b$ could not be a gap. Since $(n,b) \in E$ and $1 \leq b-n \leq m-1$, it follows that not all of $(m,m+1),\cdots,(m,2m-1)$ can lie in $E$; this implies that $n \leq 2m-1$, hence $m+3 \leq b \leq 3m-1$. The pair $(m,m+1)$ lies in $E$, so $(b-1,b)$ cannot lie in $E$, hence $b-1$ is a composite element of $S$. The only possibility is that $b = 2m+1$. Therefore $(2m-1,2m+1) \in E$. This shows that $(m,m+2) \not\in E$, so $n = m+2$. Therefore $m+2,m+3,\cdots,2m \in S$ and $2m+1 \not\in S$. The numbers $m,m+2,m+3,\cdots,2m$ generate all integers greater than $2m+1$, so $2m+1$ is the largest gap of $S$. Therefore $m = g-1$ and $S = \NGS^3_g$ in this case.
\end{proof}

\begin{rem}
It would be interesting to find a more direct connection between negative grading and the equality of the Deligne and effective weight bounds. It seems improbable that the fact that the same list of semigroups is found in both contexts is merely a combinatorial coincidence.
\end{rem}

\subsection{Semigroups of low genus}

The exact codimension of $\msg$ is known for all semigroups of genus less than or equal to $6$; in all of these cases, $\ew(S) = \codim \msg$. We now summarize where these results can be found in the literature.

Most of these loci $\msg$ have been described by Nakano; see Table 2 of \cite{nakano}\footnote{There is a typographical error in that table: the semigroup $\langle 5,6,7 \rangle$ is stated in one column to be $11$-dimensional, while the following column indicates that $\msg$ is an open subset of a $10$-dimensional weighted projective space. The second column is correct.}. Of the rows in Nakano's table where $\dim \msg$ is not known, all but one are in fact one of the negatively graded semigroups discussed in Section \ref{ss_deligne}, hence $\codim \msg$ is equal to $\ew(S)$ in those cases. The remaining semigroup is $S = \langle 5,7,9,11,13\rangle$ ($N(6)_{12}$, in the naming system of \cite{nakano}). The discussion in \cite[Section 2.2]{bullock2014} shows that for this semigroup, $\msg$ has a component of codimension $\ew(S)$ (equal to $\wt(S)$ in this case since $S$ is primitive), and the main theorem of \cite{bullock2014} shows that this is the only component.

\subsection{Two-generator semigroups} We now show a calculation showing that any numerical semigroup $S$ with only two generators exists as a Weierstrass semigroup, and that $\msg$ is irreducible of codimension $\ew(S)$ in $\cM_{g,1}$. This furnishes an infinite family of non-primitive semigroups of effective weight larger than $g$ for which the effective weight gives the correct codimension.

Let $1 < e < d$ be relatively prime integers, and let $S = \langle e,d \rangle$. The genus of $S$ is $\frac12 (e-1)(d-1)$, as a short combinatorial argument shows.

The effective weight is the number of gaps greater than $e$ plus the number of gaps greater that $d$, which can be expressed as:
$$
\ew(S) = 2g - d - e + \left\lfloor \frac{d}{e} \right\rfloor + 2.
$$

To analyze $\msg$, we use the following description.

\begin{prop}
Let $S,g,d,e$ be as above, and let $P$ denote the convex lattice polygon $\{ (i,j) \in \textbf{R}^2:\ i,j \geq 0,\ ei + dj \leq ed \}$. Let $(c_{i,j})_{(i,j) \in P}$ be coefficients such that the affine curve $\widetilde{C}$ defined by $$0 \ = \sum_{(i,j) \in P \cap \ZZ^2} c_{i,j} x^i y^j$$ is smooth, and such that the coefficients $c_{d,0}$ and $c_{0,e}$ are nonzero. Then the completion $C$ of $\widetilde{C}$ has only one additional point $p$, which has Weierstrass semigroup $S$. Viewing the coordinates $x$ and $y$ as rational functions on $C$ regular on $\widetilde{C}$, the pole orders of $x$ and $y$ at $p$ are $d$ and $e$, respectively.

Conversely, given any $(C,p) \in \msg$ and rational functions $f,g$ of pole orders $e,d$ at $p$ and regular elsewhere, the map $(f,g)$ embeds $\widetilde{C} = C \backslash \{p\}$ as an affine curve of the form above.
\end{prop}

\begin{proof}
Embed the affine plane as the set $U = \{ (x,y,1) \}$ in the weighted projective plane $\PP(e,d,1)$, and let $C$ denote the closure of $\widetilde{C}$ in $\PP(e,d,1)$. Denote by $X,Y,Z$ the weighted homogeneous coordinates on $\PP(e,d,1)$. The equation of $C$ is 
\begin{equation} \label{eq:toric}
0 = \sum_{(i,j) \in P \cap \ZZ^2} c_{i,j} X^i Y^j Z^{de-ei-dj}.
\end{equation}

Note that the non-vanishing of $c_{d,0}$ and $c_{0,e}$ ensures that the scheme cut out by this homogeneous equation has no components supported on the complement of $U$, so since this scheme matches $\widetilde{C}$ on $U$ it is indeed equal to the closure of $\widetilde{C}$.

Neither of the points $(1,0,0), (0,1,0)$ lie on $C$ since $c_{d,0}$ and $c_{0,e}$ are nonzero. Therefore any points of $C \backslash \widetilde{C}$ lie on $\{(x,y,0):\ x,y \neq 0 \} \cong \Spec k [ u,u^{-1}]$, where $u = X^d Y^{-e}$. The scheme-theoretic intersection of $C$ with this curve is given by the equation $c_{d,0} u + c_{0,e} = 0$. Hence $C$ meets the boundary transversely in a single point; it follows that $C$ is smooth, hence it is the completion of $\widetilde{C}$, and has exactly one additional point on the boundary; denote this point by $p$.

The rational functions $x$ and $y$ are regular on $\widetilde{C}$ and their divisors of zeros are degree $e$ and $d$, respectively, hence they have poles of orders $e$ and $d$ at $p$. It follows that the Weierstrass semigroup of $p$ contains $e$ and $d$, hence it contains all of $S$. It suffices to verify that the genus of $C$ is equal to the genus of $S$, which is $\frac12 (d-1)(e-1)$. This can be deduced from standard results in the geometry of toric surfaces; we summarize an argument using results from \cite{cls_toric}. To the convex lattice polygon $P$, we may associate, as described in \cite{cls_toric}, a toric variety $X_P$ together with a projective embedding. The variety $X_P$ is isomorphic to $\PP(e,d,1)$ \cite[Exercise 10.2.6(a)]{cls_toric}. The hyperplane sections in this embedding are subschemes cut out by equations of the form of Equation \ref{eq:toric}, so the curve $C$ is one such hyperplane section. By \cite[Proposition 10.5.8]{cls_toric}, the arithmetic genus of the subscheme cut out by Equation \ref{eq:toric} is equal to the number of interior lattice points of $P$. The area of $P$ is $\frac12 de$, and the number of boundary vertices of $P$ is $d+e+1$ (since $d,e$ are relatively prime, there are no lattice points interior to the edge from $(d,0)$ to $(0,e)$, so we need only count the points on the other two edges). It follows from Pick's theorem that the number of interior lattice points of $P$ is $\frac12 (d-1)(e-1)$, as desired.

For the converse, suppose that $f,g$ are rational functions on $C$ as in the Proposition statement, and let $\widetilde{C}$ be $C \backslash \{p\}$. Then $(f,g)$ defines a map from $\widetilde{C}$ to the affine plane. The ring generated by $f$ and $g$ includes functions of every possible pole order at $p$, hence this ring includes all regular functions on $\widetilde{C}$, and $(f,g)$ is an embedding. Both $f^d$ and $g^e$ have pole order $de$ at $p$, so some linear combination of them has a strictly smaller pole order, hence is expressible as a linear combination of functions $f^i g^j$, where $ei+dj \leq de$. In other words, $\widetilde{C}$ satisfies a relation of the form $0 = \sum_{(i,j) \in P \cap \ZZ^2} c_{i,j} x^i y^j$. Since there can be no relations of smaller degree, this must be the generator of the ideal of (the image in the affine plane of) $\widetilde{C}$. 
\end{proof}

We can use this description to determine the dimension of $\msg$. The dimension of the space of curves in the affine plane of the form in the Proposition is $|P \cap \ZZ^2|-1$. Using Pick's theorem and the fact that there are $d+e+1$ vertices on the boundary of $P$ (as in the proof of the Proposition), this dimension is $\frac12 (d+1)(e+1)$. This exceeds $\dim \msg$ by the dimension of the set of ways to embed a given $(C,p) \in \msg$ in this manner, which is equal to $h^0 ( \cO_C(e \cdot p)) + h^0(\cO_C(d \cdot p))$, which in turn is equal to $4 + \lfloor \frac{d}{e} \rfloor$. Therefore
\begin{eqnarray*}
\dim \msg  &=& \frac12(d+1)(e+1) - 4-  \left\lfloor \frac{d}{e} \right\rfloor\\
&=& g + d + e - 4 -  \left\lfloor \frac{d}{e} \right\rfloor.
\end{eqnarray*}

Combining with the earlier calculation of $\ew(S)$, we have proved:

\begin{prop}
Let $S = \langle e,d \rangle$ be a numerical semigroup with two generators. Then $\msg$ is irreducible of codimension $\ew(S)$ in $\cM_{g,1}$.
\end{prop}

\subsection{Total inflection points of nodal plane curves}

Another naturally arising class of semigroups for which the effective weight bound is exact are those arising from nodal plane curves. These have been investigated by Coppens and Kato \cite{ck}. Although they do not explicitly analyze the dimension of $\msg$, their results readily give its value, which coincides with the value that the effective weight would predict.

\begin{defn} \label{d_nddelta}
Let $d \geq 3$ be an integer, and $\delta$ a nonnegative integer less than $\binom{d-1}{2}$. Let $$N_{d,\delta} = \langle d-1,d \rangle \cup H_c,$$
where $g = \binom{d-1}{2} - \delta$ and $c$ is the $g$th gap in $\langle d-1,d \rangle$.\end{defn}

\begin{rem}
The genus of the semigroup $\langle d-1,d \rangle$ is $\binom{d-1}{2}$, so this is well-defined. The semigroup $N_{d,\delta}$ can be thought of as the ``simplest'' (e.g. the lowest-effective-weight) semigroup of genus $g$ containing both $d$ and $d-1$. It can also be described as the genus $g$ ancestor of $\langle d-1,d\rangle$ in the semigroup tree (see Remark \ref{r_st}).
\end{rem}

\begin{thm}[{\cite[Theorem 2.3]{ck}}] \label{t_ck}
Let $L$ be a fixed line in $\PP^2$. Let $X$ denote the variety of degree $d$ plane curves $C$ with $\delta$ simple nodes and smooth at all other points, such that $C$ intersects $L$ at a smooth point of $C$ to multiplicity $d$. Then for a general point $[C] \in X$, the Weierstrass semigroup of $(C,p)$ is $N_{d,\delta}$.
\end{thm}

\begin{prop}
For $S = N_{d,\delta}$, with $d,\delta$ as in Definition \ref{d_nddelta}, $\msg$ is irreducible of codimension $\ew(S)$ in $\cM_{g,1}$.
\end{prop}
\begin{proof}
The genus of $S$ is $g = \binom{d-1}{2} - \delta$ by definition, and its only generators that are below any gaps are $d-1$ and $d$, which lie below all gaps of $S$ except $1,2,\cdots,d-2$. Therefore $$\ew(S) = 2g-2d+4.$$
Let $X$ be the variety in the statement of Theorem \ref{t_ck}. It has a dense open subset $U$ consisting of curves $C$ such that the normalization of $(C,p)$ lies in $\msg$, and the induced map $U \rightarrow \msg$ has irreducible fibers of dimension $6$, since there is a $6$-dimensional space of automorphisms of $\PP^2$ fixing a line. The map is dominant since any $(C,p)$ with Weierstrass semigroup $S$ may be given a morphism to $\PP^2$ using two rational functions of pole orders $d-1,d$ at $p$; the image curve will be smooth at the image of $p$, and the image of $p$ will be a total inflection point since the divisor $d \cdot p$ must be the pullback of some hyperplane section. Therefore $\dim X = \dim \msg + 6$. It suffices to show that $X$ is irreducible of dimension $g + 2d$. The dimension of $X$ is equal to $g+2d$ by \cite[Lemma 2.4]{harrisSeveri}, and the irreducibility of $X$ follows from \cite[Irreducibility Theorem (bis)]{ran}.
\end{proof}

\subsection{A case where $\codim \msg \neq \ew(S)$} \label{ss_notew}

The smallest genus in which we are aware of a semigroup $S$ for which $\codim \msg \neq \ew(S)$ is $g=9$.

The example is 
\begin{eqnarray*}
S &=& \langle 6,7,8 \rangle\\
&=& \NN \backslash \{1,2,3,4,5,9,10,11,17\}.
\end{eqnarray*}

For this semigroup, $\ew(S) = 12$, but we claim that $\codim \msg =11$. We will sketch a proof of this fact, omitting the full details. In \cite{pfl}, we describe $\msg$ for all semigroups of the form $\langle d-r+1,d-r+2,\cdots,d\rangle$ in complete detail. These semigroups furnish a large collection of cases where $\codim \msg < \ew(S)$.

If $(C,p) \in \msg$, then one can show that the complete linear series $|8p|$ embeds $C$ in $\PP^3$ as the complete intersection of a quadric $Q$ and a quartic $R$, and in this embedding the osculating plane $H$ at $p$ meets $C$ at $p$ only. Hence we can study $\msg$ via the variety of triples $(C,H,p)$ of a smooth complete intersections $C$ of a quadric and quartic, a hyperplane $H$, and a point $p$ such that $C$ and $H$ meet at $p$ only. One can calculate that the dimension of this variety is $29$, and verify that for a general point of this variety, $(C,p)$ does indeed have Weierstrass semigroup $S$. Since a point $(C,p) \in \msg$ determines the triple $(C,H,p)$ up to automorphisms of $\PP^3$, this shows that $\dim \msg = 29 - \dim \textrm{Aut} \PP^3 = 14$, hence $\codim \msg = 25 - 14 = 11$.


\section{Dimensionally proper linear series} \label{s_dp}

This section collects several key facts and definitions about families of linear series on marked algebraic curves, including a ``regeneration lemma'' from the theory of limit linear series. The regeneration lemma is the basic inductive tool in the proof of Theorem \ref{t_ewexist}.

Our discussion will be brief, and a number of proofs and precise definitions are omitted where they are not necessary for the application in this paper. A complete discussion of these matters can be found in \cite[chapter 4]{ossBook}; other useful references are \cite[chapter IV]{acgh}, \cite[chapter 5]{hm} and \cite[chapter XXI]{gac2}.

\subsection{Varieties of linear series with specified ramification}

\begin{defn}
Let $C$ be a smooth curve. A \emph{linear series of rank $r$ and degree $d$} on $C$, or ``a $g^r_d$,'' is a pair $(L,V)$ consisting of a degree $d$ line bundle on $C$ and an $(r+1)$-dimensional vector space $V$ of global sections of $L$. We will sometimes refer to the linear series simply as $V$.

Let $p$ be a point of $C$. The \emph{vanishing sequence} $a_0^V(p), \cdots, a_r^V(p)$ of $V$ consists of the $r+1$ distinct orders of vanshing of elements $s \in V$ at the point $p$, in (strictly) increasing order.
\end{defn}

We will often use the phrase \emph{vanishing sequence} to refer to a set of $r+1$ nonnegative integers between $0$ and $d$ inclusive (when the values of $r,d$ are clear from context). Vanishing sequences will be denoted by capital roman letters, while the individual elements of a vanishing sequence will be denoted by the corresponding lowercase letter, with a subscript. For example, the elements of a vanishing sequence $A$ will be denoted $a_0,a_1,\cdots, a_r$, in increasing order.

In the following two definitions, we describe the set of closed points of a scheme without specifying the scheme structure. We hope the reader will forgive this, as the scheme structure is not relevant to our application. Full details, including the functors that these schemes represent, can be found in \cite[Section 4.1]{ossBook}. Although we only need the following two definitions in the cases $n=1$ and $n=2$, we state them in fuller generality.

\begin{defn}
Let $C$ be a smooth curve, $p_1,\cdots,p_n$ be distinct points of $C$, and $A^1,\cdots,A^n$ be vanishing sequences. Denote by $$G^r_d(C;\ (p_1,A^1), \cdots, (p_n,A^n))$$ a scheme whose closed points correspond to the $g^r_d$s $(L,V)$ on $C$ such that for $i=1,\cdots,n$ and $j=0,\cdots,r$, the inequality $a^V_j(p_i) \geq a^i_j$ holds (recall that we write $a^i_j$ to denote the $j$th element of the set $A^i$). Denote by $$\widetilde{G}^r_d(C;\ (p_1,A^1), \cdots, (p_n,A^n) )$$ the open subscheme where equality $a^V_j(p_i) = a^i_j$ holds for all $i,j$.
\end{defn}

\begin{rem}
In this definition and those that follow, our notation differs slightly from that of, for example, \cite{ossBook}. In particular, we specify the \emph{vanishing} sequence at each marked point, whereas most authors specify the \emph{ramification} sequence, defined by $\alpha_i(p) = a_i(p)-i$. We have chosen to work exclusively with vanishing orders, as it significantly reduces clutter in several parts of the present paper.
\end{rem}

\begin{defn}
Let $\cC \rightarrow B$ be a smooth, proper family of curves and $s_1,\cdots,s_n$ be disjoint sections. In case $n=0$, assume that the family has at least one section. Denote by $$G^r_d(\cC/B;\ (s_1,A^1),\cdots,(s_n,A^n) ) \rightarrow B$$ a scheme whose fiber over $b \in B$ is $G^r_d(\cC_b;\ (s_1(b),A^1),\cdots,(s_n(b),A^n) )$.

Denote by $\cG^r_{g,d}(A^1,\cdots,A^n) \rightarrow \cM_{g,n}$ the scheme formed by gluing these schemes together (or, more precisely, gluing these schemes together over a versal family of $n$-marked curves, and then taking the quotient by a finite group action).

The notation $\widetilde{G}^r_d$ or $\widetilde{\cG}^r_{g,d}$ will refer to the open subscheme where the vanishing sequences match the prescribed sequences exactly.
\end{defn}

Here, $\cM_{g,n}$ denotes the coarse moduli space of smooth curves with $n$ distinct marked points. We omit the details of the gluing process; it suffices for our purposes that a scheme $\cG^r_{g,d}(A^1,\cdots,A^n)$ exists, whose fibers over $\cM_{g,n}$ are isomorphic to the varieties $G^r_d(C;\ (p_1,A^1), \cdots, (p_n,A^n) )$.

\subsection{Dimensionally proper points}

\begin{defn} \label{d_rho}
For integers $g,r,d$ and vanishing sequences $A^1,\cdots,A^n$, define $$\rho_g(r,d;\ A^1,\cdots,A^n) = (r+1)(d-r) - rg - \sum_{i=1}^n \sum_{j=0}^r (a^i_j-j).$$
When $g,r,d,A^1,\cdots,A^n$ are clear from context, we will denote this number simply by $\rho$.
\end{defn}

\begin{lemma} \label{l_bndimbound}
If $G^r_d(\cC/B;\ (s_1,A^1),\cdots,(s_n,A^n))$ is nonempty, its local dimension at any point is greater than or equal to $\dim B + \rho$.
\end{lemma}

\begin{proof}
See, for example, \cite[Theorem 4.1.3]{ossBook} for full details; what follows is a brief summary. First, describe $G^r_d(\cC/B)$ (where we must assume that $\cC \rightarrow B$ has a section) as a degeneracy locus of a map of vector bundles over the relative Picard scheme $\Pic^d(\cC/B)$, and bound its dimension with this description. Note that the assumption that $\cC \rightarrow B$ has a section is needed to construct the relative Picard scheme $\Pic^d(\cC/B)$. Then impose the vanishing conditions by intersecting the pullback of $n$ Schubert cells under $n$ maps of Grassmannian bundles; this imposes at most a number of conditions equal to the double summation in the formula for $\rho$.
\end{proof}

\begin{defn} \label{d_dp}
A linear series $(L,V) \in G^r_d(C;\ (p_1,A^1),\cdots,(p_n,A^n))$ is called \emph{dimensionally proper} (with respect to the choice of $A^1,\cdots,A^n$) if there exists a deformation $(\cC/B,s_1,\cdots,s_n)$ of $(C,p_1,\cdots,p_n)$ such that $$\dim G^r_d(\cC/B;\ (s_1,A^1),\cdots,(s_n,A^n) ) = \dim B + \rho,$$ locally at $(L,V)$.

Equivalently, $(L,V)$ is dimensionally proper if the local dimension of $\cG^r_{g,d}(A^1,\cdots,A^n)$ at $(L,V)$ is equal to $3g + n -3 + \rho$.
\end{defn}

\subsection{Regeneration}

We will reduce the proof of Theorem \ref{t_ewexist} to the existence of dimensionally proper points of a suitable variety of linear series. The existence results will come from an induction on genus, made possible by the following ``regeneration lemma.'' 

\begin{lemma} \label{l_reg}
Fix positive integers $g_1,g_2,d,r$ and two vanishing sequences $A,A'$. Denote by $d-A$ the vanishing sequence $\{d-a_r,d-a_{r-1},\cdots,d-a_0\}$.

If $\widetilde{\cG}^r_{g_1,d}(A)$ and $\widetilde{\cG}^r_{g_2,d}(d-A,A')$ both have dimensionally proper points, then $\widetilde{\cG}^r_{g_1+g_2,d}(A')$ also has dimensionally proper points.
\end{lemma}

This lemma is a standard application of the theory of \emph{limit linear series}, pioneered by Eisenbud and Harris \cite{eh86}. It is essentially a special case of the ``smoothing theorem'' \cite[Theorem 3.4]{eh86}, which is referred to as the ``regeneration theorem'' in the expository account \cite[Theorem 5.41]{hm}. Both of these sources work over the complex numbers and work locally in the complex-analytic setting. We give a proof of Lemma \ref{l_reg} below based on the more recent \cite{oss06}, which is therefore valid in characteristic $p$. The theory of limit linear series has subsequently been expanded (for example, to include curves not of compact type) in various ways (e.g. \cite{oss2014,oss2014b,aminibaker}), but for our purposes the theory developed in \cite{oss06} is sufficient.

Limit linear series, as their name suggests, provide a way to construct, from a family of smooth algebraic curves degenerating to a nodal curve and a family of linear series on the smooth curves, an object over the nodal curve that serves as a well-defined limit of the the linear series on smooth curves. For the purpose of Lemma \ref{l_reg}, we need only consider particularly simple nodal curves.

\begin{sit} \label{sit_twocomp}
Fix positive integers $g_1,g_2$. Let $C_1,C_2$ be smooth curves of genus $g_1,g_2$ respectively, let $p_1$ be a point of $C_1$ and let $p_2,q$ be distinct points of $C_2$. Denote by $X$ the nodal curve obtained by gluing $p_1$ to $p_2$, and denote the attachment point by $p \in X$. See Figure \ref{fig_twocomp}.
\begin{figure}
\begin{tikzpicture}
\draw plot[smooth, tension=0.8] coordinates {(-5,1) (-4,0) (-2,-0.4) (0,0) (1,1)};
\draw plot[smooth, tension=0.8] coordinates {(-1,1) (0,0) (2,-0.4) (4,0) (5,1)};
\draw (-2,-0.5) node[below] {$C_1$};
\draw (2,-0.5) node[below] {$C_2$};
\draw[fill] (0,0) circle(2pt) node[below] {$p$};
\draw[fill] (4,0) circle(2pt) node[below] {$q$};
\end{tikzpicture}
\caption{The nodal curve $X$ of Situation \ref{sit_twocomp}.} \label{fig_twocomp}
\end{figure}
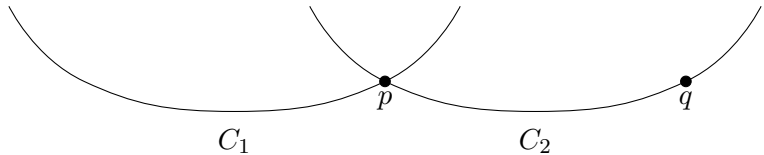
\end{sit}

We will only require a specific type of limit linear series, namely \emph{refined} series. In general, the refined series form an open subset of all limit linear series. We do not require non-refined series (called \emph{coarse} series in the Eisenbud and Harris theory) for our application, so we will not discuss them. 

\begin{defn}
In Situation \ref{sit_twocomp}, a \emph{refined limit linear series} of rank $r$ and degree $d$ on $X$ (or a \emph{limit $g^r_d$ on $X$}), is a pair $((L_1,V_1),(L_2,V_2))$ of $g^r_d$s on $C_1$ and $C_2$ respectively, such that for $i=0,1,\cdots,r$, 
\begin{equation} \label{eq_compat}
a^{V_1}_i(p_1) = d-a^{V_2}_{r-i}(p_2).
\end{equation} 
Equation \ref{eq_compat} is called the compatibility condition. The linear series $(L_i,V_i)$ is called the \emph{$C_i$-aspect} of the limit linear series.
\end{defn}

Another way to view a refined limit linear series on $X$ is that it consists of a choice of vanishing sequence $A$ (with respect to the data $r,d$) and a point in $$\widetilde{G}^r_d(C_1;\ (p_1,A)) \times \widetilde{G}^r_d(C_2;\ (p_2,d-A)).$$
Therefore a natural way to define a scheme structure for the set of refined limit linear series is as follows.
\begin{defn} \label{d_llsspaceX}
In Situation \ref{sit_twocomp}, the scheme of refined limit $g^r_d$s on $X$ is
$$
G^{r,\rf}_d(X) = \bigcup_{A} \widetilde{G}^r_d(C_1;\ (p_1,A)) \times \widetilde{G}^r_d(C_2;\ (p_2,d-A)),
$$
where the union is taken within the scheme $G^r_d(C_1) \times G^r_d(C_2)$.
\end{defn}

Definition \ref{d_llsspaceX} extends in an obvious way to families $\cX \rightarrow B$ of two-component curves. What is less obvious is that it can also be extended to certain families of curves in which some members are smooth and some are singular. For our purposes, we require the following facts.

\begin{enumerate}
\item There is a special type of family $\cX \rightarrow B$ of nodal curves, called a \emph{smoothing family} \cite[Definition 3.1]{oss06}. For every flat, proper family $\cX \rightarrow B$ of genus $g$ curves, all either smooth curves or two-component curves with one node, with $\cX$ regular and $B$ regular and connected, and every choice of point $b \in B$, there is an \'etale neighborhood $B' \rightarrow B$ of $b$ such that the fiber product $\cX' \rightarrow B'$ is a smoothing family \cite[Lemma 3.3]{oss06}.
\item If $\cX \rightarrow B$ is a smoothing family of curves, all either smooth or two-component, there is a scheme $G^{r,\rf}_d(\cX/B) \rightarrow B$, whose fiber over any $b \in B$ is either $G^r_d(\cX_b)$ (if $\cX_b$ is smooth) or $G^{r,\rf}_d(\cX_b)$ (if $\cX_b$ is a two-component curve) \cite[Proposition 6.6]{oss06}.
\item For such a smoothing family, the dimension bound $$\dim G^{r,\rf}_d(\cX/B) \geq \dim B + \rho_g(r,d)$$ holds locally at every point \cite[Theorem 5.3]{oss06}.
\item With a family $\cX \rightarrow B$ as above, given a section $s$ whose image lies in the smooth locus of every fiber, and a vanishing sequence $A$, there also exists a scheme $$\widetilde{G}^{r,\rf}_d(\cX/B;\ (s,A)) \rightarrow B,$$ whose fiber over a point $b \in B$ such that $\cX_b$ is smooth is isomorphic to $\widetilde{G}^r_d(\cX_b;\ (s(b),A))$, and whose fiber over a point $b \in B$ such that $\cX_b$ is singular consists (set-theoretically) of those refined limit linear series such that the aspect of the component on which $s(b)$ lies has vanishing sequence equal to $A$ at $s(b)$ \cite[Corollary 6.10]{oss06}.
\item In the previous situation, the dimension bound $$\dim G^{r,\rf}_d(\cX/B;\ (s,A)) \geq \dim B + \rho_g(r,d;\ A)$$ holds locally at every point \cite[Theorem 4.4.10]{ossBook}.
\end{enumerate}

With this machinery in place, we can prove the regeneration lemma.

\begin{proof}[Proof of Lemma \ref{l_reg}]
Suppose that there are two dimensionally proper linear series
\begin{eqnarray*}
(L_1,V_1) &\in& \widetilde{G}^r_d(C_1;\ (p_1,A))\\
(L_2,V_2) &\in& \widetilde{G}^r_d(C_2;\ (p_2,d-A), (q,A')),\\
\end{eqnarray*}
where $C_1,C_2$ are smooth curves of genus $g_1$ and $g_2$, respectively. Form from $(C_1,p_1)$ and $(C_2,p_2,q)$ a nodal two-component marked curve $(X,q)$ as in Situation \ref{sit_twocomp}. Let $(\cX/B,s)$ be a versal deformation of $(X,q)$, and let $\Delta \subset B$ denote the locus of singular curves, which is of codimension $1$ in $B$. We may assume (perhaps after taking a base change to an \'etale neighborhood) that $\cX /B$ is a smoothing family, and hence form the scheme $\widetilde{G}^{r,\rf}_d(\cX/B;\ (s,A'))$ of refined limit linear series. The two linear series $(L_1,V_1)$ and $(L_2,V_2)$ constitute the aspects of a refined limit linear series on $X$. Since both of these aspects are dimensionally proper, the local dimension, at this point, of the preimage of $\Delta$ in $\widetilde{G}^{r,\rf}_d(\cX/B;\ (s,A'))$ must be equal to exactly $\dim \Delta + \rho_{g_1}(r,d;\ A) + \rho_{g_2}(r,d;\ d-A,A')$. A bit of algebra shows that this is equal to $\dim \Delta + \rho_g(r,d;\ A')$ (this bit of algebra is sometimes referred to as ``the additivity of the Brill-Noether number,'' e.g. in \cite[Lemma 3.6]{eh86}). On the other hand, the local dimension, at this same point, of the entire space $\widetilde{G}^{r,\rf}_d(\cX/B;\ (s,A'))$ is \emph{at least} $\dim B + \rho_g(r,d;\ A')$; since $\Delta$ has codimension one, it follows that the local dimension of $\widetilde{G}^{r,\rf}_d(\cX/B;\ (s,A'))$ is in fact exactly equal to $\dim B + \rho_g(r,d;\ A')$, and that no irreducible component containing this point lies entirely over $\Delta$. Taking any irreducible component and restricting it to the complement of $\Delta$ in $B$, we obtain a dimensionally proper family of $g^r_d$s on \emph{smooth} marked curves of genus $g$, with imposed vanishing sequence $A'$ at the marked point. Hence $\widetilde{\cG}^r_{g,d}(A')$ has dimensionally proper points.
\end{proof}

\section{The effective weight bound} \label{s_bound}

We will prove Theorem \ref{t_ewbound} in this section. The proof comes from the dimension bound of Lemma \ref{l_bndimbound}, applied to carefully chosen vanishing data at the marked point. Our main point of departure from previous work on this subject (e.g. \cite{eh87,bullock}) is that we consider \emph{incomplete} linear series (that is, $(L,V)$ where $V$ is a strict subspace of the space of global sections of $L$), which nonetheless determine the Weierstrass semigroup. This innovation allows the weight bound to be improved to the effective weight bound.

\begin{defn} \label{d_effsubseq}
Let $S \subset \NN$ be a numerical semigroup of genus $g$. An \emph{effective subsequence for $S$} is a finite subset $T \subset S$ such that
\begin{enumerate}
\item $T$ contains $0$,
\item $T$ contains all generators of $S$, and
\item $T$ does not contain any composite elements of $S$ that are less than the largest gap of $S$.
\end{enumerate}
\end{defn}

In the statement below and elsewhere, we will write $d-T$ to denote the set $\{d-t:\ t \in T\}$.

\begin{lemma} \label{l_dimGT}
Let $T$ be an effective subsequence for a numerical semigroup $S$ of genus $g$, and $d \geq \max T$ an integer. Let $r = |T|-1$. For any smooth marked curve $(C,p)$ of genus $g$:
\begin{enumerate}
\item If the Weierstrass semigroup of $(C,p)$ is not $S$, then $$\widetilde{G}^r_d(C;\ (p,d-T)) = \emptyset.$$
\item If the Weierstrass semigroup of $(C,p)$ is $S$, then the reduced structure of $\widetilde{G}^r_d(C;\ (p,d-T))$  is isomorphic to the affine space of dimension $$\rho_g(r,d;\ d-T) + \ew(S).$$
\end{enumerate}
\end{lemma}

\begin{proof}
Suppose that $(L,V) \in \widetilde{G}^r_d(C;\ (p,d-T))$. Since $0 \in T$, one of the vanishing orders of $V$ must be $d$ itself. Therefore $L$ must be $\cO_C(d \cdot p)$, and $V$ may be regarded as a vector space of rational functions on $C$, regular away from $p$, including functions of pole orders $t \in T$ and no others. In particular, $T$ is a subset of the Weierstrass semigroup of $p$. Hence $S(C,p)$ contains all of the generators of $S$, and hence is precisely equal to $S$ since $S$ and $S(C,p)$ have the same genus. This proves part (1).

Now suppose that the Weierstrass semigroup of $(C,p)$ is $S$. Let $W$ be the vector space of global sections of $\cO_C(d \cdot p)$; regard the elements of $W$ as rational functions on $C$. This space has a complete flag $\{0\} = W_0 \subset W_1 \subset \cdots \subset W_{\ell} = W$, where $W_i$ consists of those rational functions of pole order less than $s_i$, where $S = \{0=s_0, s_1, s_2, \cdots \}$ (written in increasing order). Then the reduced structure of $\widetilde{G}^r_d(C;\ (p,d-T))$ may be identified with an open Schubert cell in the Grassmannian of $(r+1)$-dimensional subspaces of $W$ with respect to this flag, hence it is isomorphic to an affine space. If we write $T = \{s_{j_i}:\ i = 0,\cdots,r\}$ ($j_i$ increasing with $i$), then the dimension of this Schubert cell is equal to $\sum_{i=0}^r (j_i - i)$. For $i=0,1,2,\cdots,r$, $s_{j_i}-j_i$ is equal to the number of gaps below $s_{j_i}$, and therefore $$j_i - i = (s_{j_i} - i)  - g + (\# \textrm{gaps of $S$ greater than $s_{j_i}$} ).$$

Summing over all $i$ and performing some algebra, we obtain $$\dim \widetilde{G}^r_d(C;\ (p,d-T)) = \rho_g(r,d;\ d-T) - g + \sum_{t \in T} \left( \# \textrm{gaps of $S$ greater than $t$} \right).$$ Now, the value $0 \in T$ contributes $g$ to the sum on the right side of this equation, the set of generators of $S$ contribute $\ew(S)$ total to the sum, and all elements of $T$ that are composite in $S$ have no gaps of $S$ above them, thus contribute $0$. Therefore $\dim \widetilde{G}^r_d(C;\ (p,d-T)) = \rho_g(r,d; d-T) + \ew(S)$.
\end{proof}

\begin{rem} \label{rem_wtbound}
If $T$ were selected to be $S \cap \{n \in \NN:\ n \leq 2g-1\}$ (that is, if we include many composite elements), then the same proof would show that $\widetilde{G}^r_d(C;\ (p,d-T))$ is either empty or a single point, and the following corollary would prove the ordinary weight bound $\codim \msg \leq \wt(S)$. Omitting the composite elements is precisely what strengthens the bound from $\wt(S)$ to $\ew(S)$.
\end{rem}

\begin{cor} \label{c_ewboundfam}
Let $(\cC/B, s)$ be a smooth, proper family of genus $g$ curves with a section, and consider the subvariety $$B^S = \{b \in B:\ (\cC_b, s(b)) \in \msg \}$$ of marked curves with Weierstrass semigroup $S$. If $B^S$ is nonempty, then $\dim B^S \geq \dim B - \ew(S)$.
\end{cor}

\begin{proof}
The morphism $\widetilde{G}^r_d(\cC/B;\ (s,d-T)) \rightarrow B$ has image equal to $B^S$, and all fibers of dimension $\rho_g(T) + \ew(S)$. Hence $$\dim B^S = \dim \widetilde{G}^r_d(\cC/B;\ (s,d-T)) - \rho_g(r,d; d-T) - \ew(S).$$ Lemma \ref{l_bndimbound} now gives the result.
\end{proof}

\begin{prop} \label{p_epdp}
Let $S$ be a numerical semigroup, and let $T$ be an effective subsequence for $S$. Let $d$ be any integer greater than or equal to $\max T$. Also let $g$ be the genus of $S$ and $r = |T|-1$. The map $\widetilde{\cG}^r_{g,d}(d-T) \rightarrow \cM_{g,1}$ gives a bijection between the irreducible components of $\widetilde{\cG}^r_{g,d}(d-T)$ and $\msg$. Under this bijection, the effectively proper components of $\msg$ correspond to the dimensionally proper components of $\widetilde{\cG}^r_{g,d}(d-T)$.

In particular, $\msg$ has effectively proper points if and only if $\widetilde{\cG}^{r}_{g,d}(d-T)$ has dimensionally proper points.
\end{prop}
\begin{proof}
The fiber of this morphism over the point corresponding to a marked curve $(C,p)$ is equal to $\widetilde{G}^r_d(C;\ (p,d-T))$. By Lemma \ref{l_dimGT}, this fiber is either irreducible of dimension $\rho_g(r,d;\ d-T) + \ew(S)$ (if $(C,p) \in \msg$), or empty (otherwise). From this it follows that the irreducible components are in bijection, and that a component of $\msg$ has dimension $\dim \cM_{g,1} - \ew(S)$ if and only if the corresponding component of $\widetilde{\cG}^r_{g,d}(d-T)$ has dimension $\dim \cM_{g,1} + \rho_g(r,d;\ d-T)$.
\end{proof}

We can now prove Theorem \ref{t_ewbound}.

\begin{proof}[Proof of Theorem \ref{t_ewbound}]
Let $(C,p)$ be any marked smooth curve with Weierstrass semigroup $S$. Let $(\cC /B,s)$ be a versal deformation of $(C,p)$. Corollary \ref{c_ewboundfam}, applied to $(\cC/B,s)$, implies that the local dimension of $\msg$ at $(C,p)$ is at least $\dim \cM_{g,1} - \ew(S)$. For any irreducible component $X$ of $\msg$, a general point of $X$ lies on no other irreducible components, hence $\dim X \geq \dim \cM_{g,1} - \ew(S)$.
\end{proof}

\section{Secundive semigroups} \label{s_secundive}

This section collects several purely combinatorial ingredients needed to perform the inductive proof of Theorem \ref{t_ewexist}.

\begin{defn}
A numerical semigroup $S$ is called \emph{secundive} if the largest gap is smaller than the sum of the two smallest generators.
\end{defn}

\begin{rem}
The author has chosen ``secundive'' as a weaker form of ``primitive'' (``primus'' and ``secundus'' meaning, respectively, ``first'' and ``second'' in Latin).
\end{rem}

\begin{lemma} \label{l_secundive}
If $S$ is a semigroup with $\ew(S) \leq g-1$, then $S$ is secundive.
\end{lemma}

\begin{proof}
Let $S$ be a semigroup that is not secundive; we will show that $\ew(S) \geq g$. Let $m,n$ be the smallest and second-smallest generators of $S$, and let $f$ be the largest gap of $S$. Since $S$ is not secundive, $f > m+n$.

Consider the following three subsets of $\NN \times \NN$.

\begin{enumerate}
\item $\{ (m,a):\ m < a \mbox{ and } a \not\in S \}$
\item $\{ (n,a):\ n \leq a < m+n \mbox{ and } a \not\in S \}$
\item $\{ (a,f):\ n \leq a < m+n,\ m \nmid a, \mbox{ and } a \in S \}$
\end{enumerate}

These three sets are disjoint, and every pair $(x,y)$ in one of the three sets consists of a generator $x$ and a gap $y$, with $x < y$. Therefore the sum of the sizes of the three sets is less than or equal to $\ew(S)$.

The size of the first set is $g-m+1$. The sum of the sizes of the second and third sets is equal to the number of integers $a \in \{n,n+1,\cdots,m+n-1\}$ that are not divisible by $m$. There is exactly one $a$ such that $m | a$ and $n \leq a < m+n$, hence the sum of the sizes of the second and third sets is equal to $m-1$. It follows that $\ew(S) \geq g$.
\end{proof}

\begin{rem}
The method of the proof above, with slight modification, shows that the inequality $\ew(S) \geq g$ is sharp (for non-secundive semigroups), and provides a method to enumerate the equality cases. In fact, there exist non-secundive semigroups with $\ew(S) = g$ for all $g \geq 6$. On the other hand, all semigroups of genus $g \leq 5$ are secundive.
\end{rem}

Our inductive argument requires reducing the study of one secundive semigroup to another, which must be smaller both in genus and in effective weight. This is accomplished with the following operation.

\begin{defn}
For two integers $s,k$, with $k \geq 2$, define 
\begin{equation*}
\slide_k(s) = \begin{cases}
s & \mbox{ if $s \equiv 0\mod k$ }\\
s-2 & \mbox{ if $s \equiv 1\mod k$}\\
s-1 & \mbox{ otherwise.}
\end{cases}
\end{equation*}
For a set $S$ of integers and an integer $k \geq 2$, define $$\slide_k(S) = \{ \slide_k(s):\ s \in S \}$$
\end{defn}

In other words, $\slide_k$ fixes all multiples of $k$ in place, and replaces each non-multiple with the preceding non-multiple. In particular, this function is order-preserving when restricted to non-multiples of $k$; this is the feature which makes it interact well with the effective weight.

\begin{defn} \label{d_goodslider}
Let $S$ be a secundive numerical semigroup of genus $g$. Call an element $k \in S$ a \emph{good slider} if the following three conditions are met.
\begin{enumerate}[(a)]
\item $S' = \slide_k(S)$ is a secundive numerical semigroup of genus $g-1$.
\item $\ew(S') = \ew(S) -1$.
\item There exists an effective subsequence (Definition \ref{d_effsubseq}) $T$ for $S$ such that $\slide_k(T)$ is an effective subsequence for $S'$.
\end{enumerate}
\end{defn}

\begin{lemma} \label{l_slidercriteria}
Let $S$ be a secundive numerical semigroup, and let $m$ be the smallest generator of $S$.\begin{enumerate}
\item If $m+1 \not\in S$, then $m$ is a good slider.
\item If the largest gap of $S$ is less than $2m-1$, then any $k \in S$ such that $k+1 \not\in S$ is a good slider.
\item If $m \geq 3$, $2m-2 \in S$ and $2m-1$ is the largest gap of $S$, then $2m-2$ is a good slider.
\end{enumerate}
\end{lemma}

\begin{proof}
\emph{Part (1).} Suppose that $m+1 \not\in S$, and let $S' = \slide_m(S)$. Let $n$ be the second-smallest generator of $S$, and let $f$ be the largest gap of $S$; note that neither is divisible by $m$. Then $m$ is the smallest positive element of $S'$ (this is where we use the hypothesis that $m+1 \not\in S$), the smallest element of $S'$ that isn't a multiple of $m$ is $n' = \slide_m(n)$, and the largest integer that is not in $S'$ is $f' = \slide_m(f)$.

Since $S$ is secundive, $f-n \leq m-1$. Equivalently, there are fewer than $m-1$ non-multiples of $m$ between $n$ and $f$ inclusive. The same is true of $n'$ and $f'$ since sliding preserves order among non-multiples of $m$, hence $f'-n' \leq m-1$ as well.

The sum of any two elements of $S'$ is either a multiple of $m$ or exceeds $m+n'$, which exceeds $f'$, hence this sum lies in $S'$. So $S'$ is indeed a numerical semigroup. Since $m+n' > f'$, $S'$ is secundive. The gaps of $S'$ are precisely $\{ \slide_m(a):\ a \not\in S,\ a \geq 2\}$, so the genus of $S'$ is $g-1$.

To compare $\ew(S)$ and $\ew(S')$, observe first that in a secundive semigroup, an element $a$ smaller than the largest gap is a generator if and only if it is either equal to $m$ or not divisible by $m$. All other generators (those larger than the largest gap) do not contribute to the effective weight. Next observe that $m$ has one fewer gap above it in $S'$ than in $S$. Finally, note that $\slide_m$ establishes a bijection between the generators of $S$ between $m$ and $f$ exclusive and the generators of $S'$ between $m$ and $f'$ exclusive, and that the number of gaps above a given generator is preserved by this bijection. This shows that $\ew(S') = \ew(S)-1$.

For part (c) of Definition \ref{d_goodslider}, let $T$ consist of $0$ and also the smallest positive element of $S$ in each congruence class modulo $m$. This set necessarily includes all generators of $S$, and the fact that $S$ is secundive implies that any composite elements of $S$ in $T$ exceeds the largest gap, hence $T$ is an effective subsequence of $S$. The set $T' = \slide_m(T)$ is precisely equal to the set containing $0$ and the smallest positive element of $S'$ in each congruence class modulo $m$, so since $S'$ is also a secundive semigroup, $T'$ is an effective subsequence of $S'$ by the same reasoning. This completes the proof that $m$ is a good slider when $m+1 \not\in S$.

\emph{Part (2).} Now assume that the largest gap of $S$ is less than $2m-1$, and that $k \in S$ is an element with $k+1 \not\in S$. Then $S$ is primitive. The smallest positive element of $S'$ is either $m-1$ or $m$, and the largest gap of $S'$ is less than $2m-2$, hence $S'$ is in fact a \emph{primitive} semigroup as well. The operation $\slide_k$ preserves the number of gaps above every element of $S$, except in one case: the number of gaps of $S'$ above $\slide_k(k) = k$ is one less than the number of gaps above $k$ in $S$. So $\ew(S') = \ew(S) -1$. Finally, the set $T$ can be constructed in a manner similar to in Part (1): let $T$ consist of $0$ and the smallest positive element of $S$ in each congruence class modulo $k$. Then $T' = \slide_k(T)$ is the result of an identical construction applied to $S'$, and a set constructed this way contains all generators of the semigroup. Since $S$ and $S'$ are primitive, the sets $T$ and $T'$ are effective subsequences, since the condition of containing no composite elements less than the largest gap is vacuous.

\emph{Part (3).} Now assume that $2m-2 \in S$, $2m-1 \not\in S$, and all integers larger than $2m$ are in $S$. Then again, $S$ is primitive. The largest gap of $S' = \slide_{2m-2}(S)$ is $2m-3$, and the smallest element of $S'$ is $m-1$ (note that $m \neq 2m-2$ since we are assuming $m \geq 3$), so $S'$ is also a primitive semigroup. The rest of the argument is now analogous to the proof of Part (2).
\end{proof}

\begin{lemma} \label{l_slidersexist}
Let $S$ be a secundive numerical semigroup, and let $m$ be the smallest generator of $S$. If $m+1 \in S, 2m-2 \not\in S$, and $2m-1 \not\in S$, then $\ew(S) \geq g-1$. Furthermore, equality $\ew(S) = g-1$ occurs if and only if $S = \{0,m,m+1\} \cup H_{2m}$.
\end{lemma}

\begin{proof}
Note that the hypotheses imply that $m \geq 4$, so $m+1 < 2m-2$. Also note that since $S$ is secundive and contains $m+1$, in fact $S$ is primitive and the largest gap is $2m-1$.

Since all elements of $S$ less than the largest gap are generators, the effective weight (which is equal to the weight) is equal to the size of the set $E = \{(a,b) \in \NN^2:\ 0<a<b,\ a\in S,\ b \not\in S\}$.

The elements of $E$ can be partitioned into four types.
\begin{enumerate}
\item The pairs $(m,2m-2),(m,2m-1),(m+1,2m-2)$, and $(m+1,2m-1)$.
\item Pairs of the form $(m,a)$ or $(m+1,a)$, where $m+2 \leq a \leq 2m-3$ and $a \not\in S$.
\item Pairs of the form $(a,2m-2)$ or $(a,2m-1)$, where $m+2 \leq a \leq 2m-3$ and $a \in S$.
\item Pairs of the form $(a,b)$, where $m+2 \leq a < b \leq 2m-3$, $a \in S$, and $b \not\in S$.
\end{enumerate}

There are four pairs of the first type. Since every element $a$ between $m+2$ and $2m-3$ inclusive appears in either two pairs of the second type or two pairs of the third type (depending on whether or not $a \in S$), the total number of pairs of either the second or third type is exactly $2(m-4)$. Therefore, adding the four pairs of the first type, $\ew(S)$ is equal to $2m-4$ plus the number of pairs of the fourth type. On the other hand, the genus of $S$ is at most $(m-1) + (m-2) = 2m-3$, with equality if and only if $S$ contains no elements between $m+2$ and $2m-3$ inclusive. Hence $g-1 \leq 2m-4 \leq \ew(S)$, with equality throughout if and only if $S$ constists precisely of $0,m,m+1$ and all integers greater than or equal to $2m$.
\end{proof}

\begin{cor} \label{c_slidersexist}
If $S$ is a numerical semigroup with $1 \leq \ew(S) \leq g-2$, then $S$ has a good slider. If $\ew(S) = g-1$, then $S$ has a good slider unless $S = \{0,m,m+1\} \cup H_{2m}$ for some $m \geq 4$.
\end{cor}

\begin{proof}
Suppose that $1 \leq \ew(S) \leq g-1$ and that $S$ does not have a good slider. By Lemma \ref{l_secundive}, $S$ is secundive. Let $m$ be the smallest generator of $S$. By Lemma \ref{l_slidercriteria}(1), $m+1  \in S$; this must be the second-smallest generator. Since $S$ is secundive, the largest gap of $S$ is less than $m + (m+1)$, so it is at most $2m-1$. Since $\ew(S) > 0$, the largest gap is greater than $m$, hence at least $m+2$. Therefore $m+2 \leq 2m-1$, so $m \geq 3$. By Lemma \ref{l_slidercriteria}(2), the largest gap is in fact equal to $2m-1$, and by Lemma \ref{l_slidercriteria}(3), $2m-2 \not\in S$. This implies that $m+1 < 2m-2$, hence $m \geq 4$. By Lemma \ref{l_slidersexist}, it follows that $\ew(S)$ is equal to $g-1$ and $S = \{0,m,m+1\} \cup H_{2m}$.
\end{proof}

\section{Existence of effectively proper points} \label{s_exist}

We can now prove Theorem \ref{t_ewexist} by assembling the ingredients of the previous sections and the following statement about elliptic curves. This lemma is similar to \cite[Proposition 5.2]{eh87}, and plays an analogous role in our argument.

\begin{lemma} \label{l_ellipticbridge}
Fix integers $d,r$, and let $T,T'$ be two vanishing sequences. As usual, denote the elements of these, in increasing order, by $t_i$ and $t_i'$. Suppose that there exists an integer $k$, $1 \leq k \leq r$, such that:
\begin{enumerate}
\item $t_0 = t'_0 = 0$ and $t_k = t'_k$;
\item for all $i \not\in \{0,k\}$, neither $t_i$ nor $t'_i$ is divisible by $t_k$;
\item for all $i \not\in \{0,k\}$, the inequalities $t_{i-1} \leq t'_i < t_i$ hold.
\end{enumerate}
Then $\widetilde{\cG}^r_{1,d}(T',d-T)$ has dimensionally proper points.
\end{lemma}

\begin{proof}
Denote the number $t_k = t'_k$ by $m$. Fix an elliptic curve $E$ with a point $p$. Consider the trivial family $E \times E \rightarrow E$ given by projection to the second coordinate, with two sections $s_1(q) = (p,q)$ and $s_2(q) = (q,q)$. Fix a point $q_0 \in E$ differing from $p$ by torsion of order exactly $m$. Let $B$ be the open subset of $E$ given by removing $p$ and all points $q'$ differing from $p$ by torsion of order dividing $m$, except the point $q_0$ itself. We can now regard $(E \times B, s_1, s_2)$ as a family of twice-marked elliptic curves $\{(E,p,q)\}_{q \in B}$, with the property that exactly one member $(E,p,q_0)$ of this family has $p-q$ of torsion order $m$, and all others have $p-q$ either non-torsion or torsion of order not dividing $m$. We will show that $\widetilde{G}^r_d(E \times B / B;\ (s_1,T'),(s_2,d-T))$ is nonempty of dimension $\rho_1(r,d;\ T',d-T) + 1$, which will prove the result.

More specifically, we will show that if $p,q$ are points not differing by torsion of order dividing $m$ (which is the case for all but one member of the family), $\widetilde{G}^r_d(E;\ (p,T'),(q,d-T))$ is empty, while if $p,q$ differ by torsion of order exactly $m$ (the case for one member of the family), then $\widetilde{G}^r_d(E;\ (p,T'),(q,d-T))$ is nonempty of dimension $\rho_1(r,d;\ T',d-T)+1$.

Suppose that $(E,p,q)$ is a twice-marked elliptic curve, and $(L,V)$ is some linear series with vanishing orders exactly $T'$ at $p$ and $d-T$ at $q$.

The key observation is that for any $i \in \{0,1,\cdots,r\}$, the subspace $$V_i = V(-t'_{i}p - (d-t_i)q)$$ of $V$ consisting of sections vanishing to order at least $t'_{i}$ at $p$ and order at least $d-t_i$ at $q$ must be at least $1$-dimensional.

In particular, for $i = 0$ and $i=k$ it follows that the divisors $d\cdot q$ and $m \cdot p + (d-m) \cdot q$ are both in the divisor class defined by the line bundle $L$. Hence $L$ must be the line bundle $\cO_E(d \cdot q)$, and the points $p$ and $q$ must differ by an element of $\Pic^0(E)$ of order dividing $m$. This shows that $G^r_d(E;\ (p,T'),(q,d-T))$ is indeed empty whenever $p,q$ do not differ by torsion of order dividing $m$.

\emph{We will now assume that $p$ and $q$ differ by torsion of order exactly $m$.}

\emph{Claim 1.} For $i=0,1,\cdots,r-1$, there exist no sections $s \in V$ whose divisor of zeros contains $t'_{i+1} p + (d-t_i)q$. 

\emph{Proof of claim 1.} If $i = 0,k-1$, or $k$, then the divisor  $t'_{i+1} p + (d-t_i)q$ has degree greater than $d$, so it certainly cannot be contained in the divisor of zeros of $s$. Otherwise, $t'_{i+1} + (d-t_i) \geq d$, so the only way for the divisor of $s$ to contain such a divisor is if $t'_{i+1} = t_i$ and the divisor of $s$ is exactly $t_i p + (d-t_i) q$. But this implies that $p-q$ is $t_i$-torsion, which is impossible since $m \nmid t_i$ when $i \neq 0,k$.

\emph{Claim 2.} The space $V_i$ is exactly $1$-dimensional, and a nonzero section of $V_i$ vanishes to order exactly $t'_i$ at $p$ and $d-t_i$ at $q$.

\emph{Proof of claim 2.} The second statement follows from claim 1. The first part follows from the second: in a 2-dimensional space of sections, there must be $2$ distinct orders of vanising at any given point.

Therefore we see that $(L,V)$ has a very simple form: $L = \cO_E(d \cdot q)$ and $V$ is the span of $r+1$ disjoint $1$-dimensional subspaces $V_i$, each of which is spanned by a section of $L$ vanishing along the divisor $t'_i p + (d-t_i) q$. Conversely, it is clear that any choice of these $r+1$ spaces $V_i$ gives rise to a point of $\widetilde{G}^r_d(E;\ (p,T'),(q,d-T))$. From this description, we can calculate from Riemann-Roch:

\begin{eqnarray*}
\dim \widetilde{G}^r_d(E;\ (p,T'),(q,d-T)) &=& \sum_{i=0}^r \dim \PP H^0(L(-t'_i p - (d-t_i) q))\\
&=& \sum_{i=0}^r \dim \PP H^0( \cO_E(-t'_i p + t_i q))\\
&=& 2 + \sum_{i=0}^r (t_i - t'_i - 1).
\end{eqnarray*}
In the third line, we use the fact that for all $i \not\in \{0,k\}$, $t_i-t'_i > 0$, hence $h^1(E,\cO_E(-t'_ip+t_iq)) = 0$, while for $i=0$ and $i=k$, $h^1( \cO_E(-t'_i p + t_i q) ) = h^1( \cO_E ) = 1$.

On the other hand, a bit of algebra shows that $$\rho_1(r,d;\ T',d-T) = 1 + \sum_{i=0}^r (t_i -t'_i - 1).$$ So we have established that, in the case where $p-q$ differ by torsion of order $m$,
$$\dim G^r_d(E;\ (p,T'),(q,d-T)) = 1 + \rho_1(r,d;\ T', d-T).$$
By the remarks in the first paragraph, this proves that $\widetilde{\cG}^r_{1,d}(T',d-T)$ has dimensionally proper points.
\end{proof}

\begin{cor} \label{c_induction}
If $S$ is a secundive numerical semigroup, $k$ is a good slider for $S$, and $\cM^{\slide_k(S)}_{g-1,1}$ has effectively proper points, then $\msg$ has effectively proper points.
\end{cor}

\begin{proof}
Let $T$ be an effective subsequence of $S$ such that $T' = \slide_k(T)$ is an effective subsequence of $S' = \slide_k(S)$. Removing some elements if necessary, we may assume that $T$ and $T'$ contain no multiples of $k$ other than $0$ and $k$. Let $r = |T|-1$ and let $d = \max T$. By Proposition \ref{p_epdp}, $\widetilde{\cG}^r_{g-1,d}(d-T')$ has dimensionally proper points. By Lemma \ref{l_ellipticbridge}, $\widetilde{\cG}^r_{1,d}(T',d-T)$ also has dimensionally proper points. The regeneration lemma, Lemma \ref{l_reg}, implies that $\widetilde{\cG}^r_{g,d}(d-T)$ also has dimensionally proper points; Proposition \ref{p_epdp} now implies that $\msg$ has an effectively proper component.
\end{proof}

\begin{rem}
In fact, the proof of Corollary \ref{c_induction} shows that the existence of effectively proper points of $\cM_{g,1}^S$ can be deduced from the existence of effectively proper points of $\cM_{g-1,1}^{S'}$ whenever $S',S$ are semigroups of genus $g-1$ and $g$ respectively possessing effective subsequences $T',T$ that satisfy the hypotheses of Lemma \ref{l_ellipticbridge}. It is possible that more $\cM_{g,1}^S$ can be shown to have effective proper points by constructing $S'$ in a different way from the slide construction.
\end{rem}

\begin{proof}[Proof of Theorem \ref{t_ewexist}]
Let $S$ be a numerical semigroup of genus $g$, such that $\ew(S) \leq g-2$. We will prove that $\msg$ has effectively proper components by induction on $g$. For $g=1$ the only semigroup is $H_1$, and there is nothing to prove.

Suppose that $g \geq 2$ and the result holds for genus $g-1$. If $\ew(S) = 0$, then $S = H_g$ and $\msg$ is a dense open subset of $\cM_{g,1}$, so the result follows. Otherwise, Corollaries \ref{c_slidersexist} and \ref{c_induction} show that the existence of an effectively proper point of $\msg$ follows from the existence of an effectively proper point of $\cM^{S'}_{g-1,1}$ for some semigroup $S'$ of genus $g-1$ and effective weight $\ew(S)-1 \leq g-3$. This completes the induction.

Now suppose that $\chr k = 0$ and $S$ is a numerical semigroup of genus $g$ such that $\ew(S) = g-1$. The argument above works without modification, except in one case: $g$ is odd and $S = \{ 0, \frac12g + \frac32, \frac12g+\frac52 \} \cup H_{g+3}$ (this is the exception in Corollary \ref{c_slidersexist}). The main theorem of \cite{kom91} is that for this specific semigroup, in characteristic $0$, $\msg$ has dimensionally proper points (which are the same as effectively proper points, since $S$ is primitive). With this possibility accounted for, the induction is complete in the case $\ew(S) = g-1$ as well.
\end{proof}


\bibliography{main}{}

\begin{thebibliography}{ACGH85}

\bibitem[AB15]{aminibaker}
Omid Amini and Matthew Baker.
\newblock Linear series on metrized complexes of algebraic curves.
\newblock {\em Mathematische Annalen}, 362(1-2):55--106, 2015.

\bibitem[ACG11]{gac2}
E.~Arbarello, M.~Cornalba, and P.~Griffiths.
\newblock {\em Geometry of algebraic curves, {V}olume II}.
\newblock Springer Science \& Business Media, 2011.

\bibitem[ACGH85]{acgh}
E.~Arbarello, M.~Cornalba, P.~Griffiths, and J.~Harris.
\newblock {\em Geometry of algebraic curves, {V}olume {I}}.
\newblock Springer-Verlag, 1985.

\bibitem[BAB09]{bras}
Maria Bras-Amor{\'o}s and Stanislav Bulygin.
\newblock Towards a better understanding of the semigroup tree.
\newblock {\em Semigroup Forum}, 79(3):561--574, 2009.

\bibitem[Bul13]{bullock}
Evan Bullock.
\newblock Subcanonical points on algebraic curves.
\newblock {\em Transactions of the {A}merican {M}athematical {S}ociety},
  365(1):99--122, 2013.

\bibitem[Bul14]{bullock2014}
Evan Bullock.
\newblock Irreducibility and stable rationality of the loci of curves of genus
  at most six with a marked {W}eierstrass point.
\newblock {\em Proceedings of the American Mathematical Society},
  142(4):1121--1132, 2014.

\bibitem[CK94]{ck}
Marc Coppens and Takao Kato.
\newblock Weierstrass gap sequence at total inflection points of nodal plane
  curves.
\newblock {\em Tsukuba journal of mathematics}, 18(1):119--129, 1994.

\bibitem[CLS11]{cls_toric}
David~A. Cox, John~B. Little, and Henry~K. Schenck.
\newblock {\em Toric varieties}, volume 124 of {\em Graduate Studies in
  Mathematics}.
\newblock American Mathematical Society, Providence, RI, 2011.

\bibitem[Del73]{deligne}
Pierre Deligne.
\newblock Quadriques.
\newblock In {\em Groupes de Monodromie en G{\'e}om{\'e}trie Alg{\'e}brique},
  pages 62--81. Springer, 1973.

\bibitem[EH86]{eh86}
David Eisenbud and Joe Harris.
\newblock Limit linear series: basic theory.
\newblock {\em Inventiones mathematicae}, 85(2):337--371, 1986.

\bibitem[EH87]{eh87}
David Eisenbud and Joe Harris.
\newblock Existence, decomposition, and limits of certain {W}eierstrass points.
\newblock {\em Inventiones mathematicae}, 87(3):495--515, 1987.

\bibitem[Ful93]{fulton}
William Fulton.
\newblock {\em Introduction to toric varieties}.
\newblock Princeton University Press, 1993.

\bibitem[Har86]{harrisSeveri}
Joe Harris.
\newblock On the {S}everi problem.
\newblock {\em Inventiones mathematicae}, 84(3):445--461, 1986.

\bibitem[HM06]{hm}
Joe Harris and Ian Morrison.
\newblock {\em Moduli of curves}.
\newblock Springer Science \& Business Media, 2006.

\bibitem[Hur92]{hur}
Adolf Hurwitz.
\newblock {\"U}ber algebraische gebilde mit eindeutigen transformationen in
  sich.
\newblock {\em Mathematische Annalen}, 41(3):403--442, 1892.

\bibitem[Kom91]{kom91}
Jiryo Komeda.
\newblock On primitive {S}chubert indices of genus g and weight g-1.
\newblock {\em Journal of the Mathematical Society of {J}apan}, 43(3):437--445,
  1991.

\bibitem[KY13]{kaplanye}
Nathan Kaplan and Lynnelle Ye.
\newblock The proportion of {W}eierstrass semigroups.
\newblock {\em Journal of Algebra}, 373:377--391, 2013.

\bibitem[KZ03]{kz}
Maxim Kontsevich and Anton Zorich.
\newblock Connected components of the moduli spaces of abelian differentials
  with prescribed singularities.
\newblock {\em Inventiones mathematicae}, 153(3):631--678, 2003.

\bibitem[Nak08]{nakano}
Tetsuo Nakano.
\newblock On the moduli space of pointed algebraic curves of low genus
  {II}---rationality---.
\newblock {\em Tokyo Journal of Mathematics}, 31(1):147--160, 2008.

\bibitem[Oss]{ossBook}
B.~Osserman.
\newblock {\em Limit linear series}.
\newblock Draft monograph.

\bibitem[Oss06]{oss06}
Brian Osserman.
\newblock A limit linear series moduli scheme.
\newblock {\em Annales de l'institut {F}ourier}, 56(4):1165--1205, 2006.

\bibitem[Oss14a]{oss2014b}
Brian Osserman.
\newblock Limit linear series for curves not of compact type.
\newblock {\em arXiv preprint arXiv:1406.6699}, 2014.

\bibitem[Oss14b]{oss2014}
Brian Osserman.
\newblock Limit linear series moduli stacks in higher rank.
\newblock {\em arXiv preprint arXiv:1405.2937}, 2014.

\bibitem[Pfl16]{pfl}
Nathan Pflueger.
\newblock Weierstrass semigroups on {C}astelnuovo curves.
\newblock {\em arXiv preprint arXiv:1608.08178}, 2016.

\bibitem[Pin74]{pinkham}
Henry~C. Pinkham.
\newblock {\em Deformations of algebraic varieties with $G_m$ action}.
\newblock Soci{\'e}t{\'e} math{\'e}matique de France, 1974.

\bibitem[Ran89]{ran}
Ziv Ran.
\newblock Families of plane curves and their limits: {E}nriques' conjecture and
  beyond.
\newblock {\em Annals of Mathematics}, 130(1):121--157, 1989.

\bibitem[RV77]{rv}
Dock~Sang Rim and Marie~A Vitulli.
\newblock Weierstrass points and monomial curves.
\newblock {\em Journal of Algebra}, 48(2):454--476, 1977.

\bibitem[Zha13]{zhai}
Alex Zhai.
\newblock Fibonacci-like growth of numerical semigroups of a given genus.
\newblock {\em Semigroup Forum}, 86(3):634--662, 2013.

\end{thebibliography}
\bibliographystyle{alpha}

\end{document}